\documentclass[journal]{IEEEtran}

\usepackage{mathptmx} 
\usepackage{amsmath} 
\usepackage{amssymb}  
\usepackage{amsthm}
\usepackage{graphicx}
\usepackage{dsfont}
\usepackage{mathtools}
\usepackage{framed}
\usepackage{url}
\usepackage{pstool}
\usepackage[utf8]{inputenx}

\newtheorem{theorem}{Theorem}
\newtheorem{lemma}{Lemma}

\newtheorem{corollary}{Corollary}
\newtheorem{remark}{Remark}
\newtheorem{assumption}{Assumption}

\ifCLASSINFOpdf
\else
\fi
\begin{document}
%
\title{Multipliers for nonlinearities with monotone bounds}
%
%
%

\author{William P. Heath,~\IEEEmembership{Member,~IEEE,}
        Joaquin Carrasco,~\IEEEmembership{Member,~IEEE,}
        and~Dmitry A. Altshuller,
\thanks{W. P. Heath and J. Carrasco are with the Control Systems Centre, Department
of Electrical and Electronic Engineering, University of Manchester, UK. e-mail: william.heath@manchester.ac.uk, joaquin.carrasco@manchester.ac.uk.
The late D. A. Altshuller was with Dassault Systems, USA.}
}

\maketitle

\begin{abstract}
 We consider Lurye (sometimes written Lur'e) systems whose nonlinear operator is characterised by a possibly multivalued nonlinearity that  is bounded above and below by monotone functions. Stability can be established using a sub-class of the Zames-Falb multipliers. The result generalises similar approaches in the literature. 
Appropriate multipliers can be found using convex searches. Because the multipliers can be used for multivalued nonlinearities they can be applied after loop transformation. We illustrate the power of the new mutlipliers with two examples, one in continuous time and one in discrete time: in the first the approach is shown to outperform available stability tests in the literature; in the second  we focus on the special case for asymmetric saturation with important consequences for systems with non-zero steady state exogenous signals. 
\end{abstract}

\begin{IEEEkeywords}
Lure systems, quasi-monotone, quasi-odd, asymmetry, Zames-Falb multiplier.
\end{IEEEkeywords}

%
\IEEEpeerreviewmaketitle

\section{Introduction}
%
%
%
%

\IEEEPARstart{W}{e} are concerned with the input-output stability of the Lurye system given by
\begin{equation}
y_1=Gu_1,\mbox{ } y_2=\phi u_2,\mbox{ } u_1=r_1-y_2 \mbox{ and }u_2 = y_1+r_2.\label{eq:Lurye}
\end{equation}
 Let $\mathcal{L}_2$ be the space of finite energy Lebesgue integrable signals and let $\mathcal{L}_{2e}$ be the corresponding extended space (see for example  \cite{desoer75}). The Lurye system is said to be stable if $r_1,r_2\in\mathcal{L}_2\Rightarrow u_1,u_2,y_1,y_2\in\mathcal{L}_2$.
\begin{assumption}\label{ass0}
The  Lurye system~(\ref{eq:Lurye}) is assumed to be well-posed with $G:\mathcal{L}_{2e}\rightarrow\mathcal{L}_{2e}$  linear time invariant (LTI) causal and stable and with $\phi:\mathcal{L}_{2e}\rightarrow\mathcal{L}_{2e}$ some nonlinear operator.
\end{assumption}

A function $\alpha:\mathbb{R}\rightarrow\mathbb{R}$ is said to be monotone if $\alpha(x_1)\geq\alpha(x_2)$ for all $x_1\geq x_2$. It is  said to be  bounded  if there exists $C\geq 0$ such that $|\alpha(x)|\leq C|x|$ for all $x\in\mathbb{R}$\footnote{Here the term ``bounded'' is not used in the standard sense it is used for functions (e.g. \cite{Rudin:53}); rather it is used in a sense consistent with the notion of bounded operators (e.g.\cite{Kolmogorov:57}). We use the terms ``bounded below'' and ``bounded above'' in a further sense below.}. It is said to be odd if $\alpha(-x)=-\alpha(x)$ for all $x\in\mathbb{R}$. It is said to be slope-restricted on $[0,s]$ if $0\leq (\alpha(x_1)-\alpha(x_2))/(x_1-x_2)\leq s$ for all $x_1\neq x_2$. If the nonlinear operator $\phi$ can be characterised by the monotone and bounded function $\alpha:\mathbb{R}\rightarrow\mathbb{R}$  in the sense that
$y(t)  \triangleq(\phi u) (t) = \alpha(u(t))$,
then the Zames-Falb multipliers  may be used to determine stability \cite{OShea67,Zames68,desoer75,Carrasco:EJC}.  We call $\alpha$ the nonlinearity that characterises the nonlinear operator $\phi$. Further results may be obtained if the nonlinearity is odd, if it is slope-restricted or if it is both odd and slope-restricted \cite{Zames68,desoer75}.


\begin{figure}[t]
	\begin{center}
	\includegraphics[width = 0.8\linewidth]{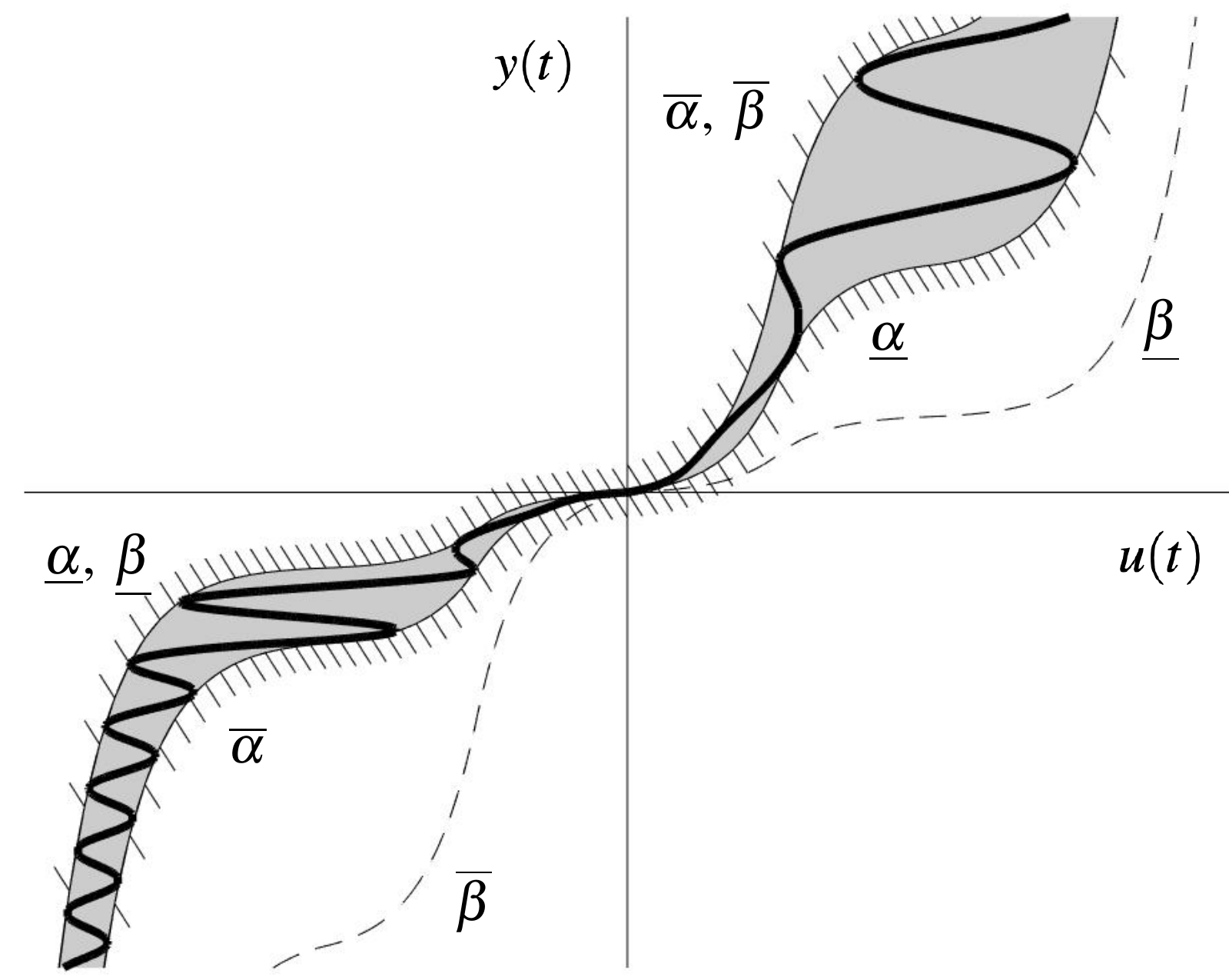}
\caption{The nonlinearity (i.e. the map from $u(t)$ to $y(t)=(\phi u)(t)$) is  bounded below and above by the monotone and bounded functions $\underline{\alpha}$ and  $\overline{\alpha}$ respectively. In addition the nonlinearity is bounded below and above by the monotone, bounded and odd functions $\underline{\beta}$  and  $\overline{\beta}$ respectively. For the specific case illustrated the functions $\underline{\beta}$ and $\overline{\beta}$ are constructed as follows:  when $u(t)<0$ set $\underline{\beta}(u(t))=\underline{\alpha}(u(t))$ and $\overline{\beta}(u(t))=-\overline{\alpha}(-u(t))$;  when $u(t)\geq 0$ set   $\underline{\beta}(u(t))=-\underline{\alpha}(-u(t))$ and $\overline{\beta}(u(t))=\overline{\alpha}(u(t))$. 
}\label{bd1}
	\end{center}
\end{figure}

In this note we consider a more general class of nonlinear operators 
where the nonlinearity need  be neither  monotone nor  a single-valued function. 
Instead, we say the nonlinear operator is characterised by a nonlinearity that is bounded below and above in the following sense.
\begin{assumption}[Fig~\ref{bd1}]\label{ass1}
Let $y(t)=(\phi u)(t)$. If $u(t)=0$ then $y(t)=0$. There are assumed to exist monotone and bounded functions $\underline{\alpha}:\mathbb{R}\rightarrow\mathbb{R}$ and $\overline{\alpha}:\mathbb{R}\rightarrow\mathbb{R}$ such that
\begin{equation}
0   \leq \frac{\underline{\alpha}(u(t))}{u(t)} \leq \frac{y(t)}{u(t)} \leq \frac{\overline{\alpha}(u(t))}{u(t)}\mbox{ for all }   u(t)\neq 0.
	\label{bound_def}
\end{equation}
We say the nonlinearity is bounded below by $\underline{\alpha}$ and above by $\overline{\alpha}$.  
There are also assumed to exist monotone, bounded and odd  functions $\underline{\beta}:\mathbb{R}\rightarrow\mathbb{R}$ and $\overline{\beta}:\mathbb{R}\rightarrow\mathbb{R}$ such that the nonlinearity is bounded below by $\underline{\beta}$ and above by $\overline{\beta}$.
\end{assumption}
\begin{remark}
For a given $u\in\mathcal{L}_{2e}$ the values of $y(t)$ remain uniquely determined even though the characterising nonlinearity may be multivalued. The Lurye problem is sometimes restricted to memoryless, but possibly time-varying, nonlinearities \cite{Vidyasagar78}. Our class   includes both dynamic and time-varying operators; nevertheless the conditions of Assumption~\ref{ass1} exclude many common nonlinearities. In the terminology of \cite{Zames66} condition (\ref{bound_def}) is an instantaneous condition. 
\end{remark}
We will often make the following further assumption.
\begin{assumption}\label{ass1a}  
Given Assumption~\ref{ass1}, there is assumed to be some finite $A\geq 1$ such that  $\overline{\alpha}$ is bounded above by $A\underline{\alpha}$. Similarly,  there is assumed to be some (possibly infinite)  $B\geq A$ such that  $\overline{\beta}$ is bounded above by $B\underline{\beta}$.
\end{assumption}


If $A=1$ then the nonlinearity is single-valued and monotone.  If $A=B=1$ then the nonlinearity is  odd. If $A=1$ and  $1<B<\infty$ we say the nonlinearity is quasi-odd. 

If $1<A<\infty$ we say the nonlinearity is quasi-monotone. If $1<A=B<\infty$ we say the nonolinearity is quasi-monotone with odd bounds.  If $1<A<B<\infty$ we say the nonolinearity is quasi-monotone with quasi-odd bounds.


\begin{remark}
In the literature \cite{Altshuller:13,Barabanov:00} ``quasimonotone''  has a wider and more general definition than that for quasi-monotone adopted here. In \cite{Materassi11}  ``quasi-monotone-and-odd'' is used for the case where, in our terminology, the nonlinearity is quasi-monotone with odd bounds. Our terminology is slightly different to that in \cite{HeathCDC15}.
\end{remark}

Our main result, Theorem~\ref{thm1}, is to derive a subclass of the Zames-Falb multipliers  that preserves the positivity of such nonlinearities. The original results of Zames and Falb \cite{Zames68}  for nonlinearities characterisd by either monotone and bounded or monotone, bounded and odd functions can be recovered as special cases  with $A=1$ and respectively where either $B$ is ignored  or $B=1$. A generalisation for quasi-odd nonlinearities follows immediately (Corollary~\ref{cor3}).

A similar approach is taken by \cite{Rantzer2001} and \cite{Materassi11} for quasi-monotone nonlinearities with odd bounds.  In \cite{Rantzer2001}  a specific stiction model is considered. Corollary~\ref{cor4} generalises the results of \cite{Rantzer2001} in two senses: firstly it allows more general bounds on the nonlinearity; secondly it allows the nonlinearity to be multivalued.
For the specific application of \cite{Rantzer2001} our results are the same. Corollary~\ref{cor4} provides a less conservative result than that of \cite{Materassi11}; a similar improvement is noted by \cite{Altshuller:13} without proof,  where the result is generalised to time-periodic, but not more general, nonlinearities. 


We extend our results to the case where the bounds on the nonlinearity are also slope-restricted (Theorem~\ref{thm2}) by applying  loop transformation techniques. A single-valued function need not be single-valued after loop transformation  (see~\cite{desoer75}) so our relaxation of the standard assumption that the nonlinearity be single-valued \cite{Zames68,Rantzer2001,Materassi11} is necessary. Once again the original results of \cite{Zames68} can be derived as special cases and we state the counterpart of Corollary~\ref{cor3} under loop transformation as Corollary~\ref{cor5}.


Our development is for continuous time multipliers. Corresponding results for discrete-time systems can be derived similarly and are briefly stated in Appendix~A.
In Appendix~B we show how the convex search for Zames-Falb multipliers \cite{Chen:95} and for their discrete-time counterparts \cite{Wang:TAC} can be modified to search for the multipliers of this paper.

We illustrate the stability results with two examples. The first, in continuous time, is similar (though not identical) to an example in \cite{Materassi11}.  It illustrates some of the subtleties that arise with loop transformations and how the new stability criteria can provide better results than those in the literature. The second example, in discrete time, illustrates how the new results can be applied to Lurye systems with asymmetric saturation. This offers insight to the behaviour of (for example) anti-windup systems with exogenous signals with non-zero steady state values. This example was discussed in \cite{HeathCDC15} where
some technical results were also presented without proof.


 


\section{Multipliers}\label{sec_gen}

In our development we will exploit the Jordan decomposition \cite{Billingsley95} of a signal. If $x\in\mathcal{L}_{2e}$ its Jordan decomposition is $x=x_+-x_-$ where $x_+(t) = \max{(x(t),0)}$ for all $t\in\mathbb{R}$. We begin by establishing the following inequalities.
\begin{lemma}\label{lem1}
Under the conditions of Assumptions~\ref{ass1} and~\ref{ass1a},
 if $\phi:\mathcal{L}_{2e}\rightarrow\mathcal{L}_{2e}$  then for all $u\in\mathcal{L}_2$ and for all $\tau\in\mathbb{R}$, 
		\begin{equation}
			-B \int_{-\infty}^{\infty} u(t)y(t)\,dt\leq \int_{-\infty}^{\infty}u(t+\tau) y(t)\,dt   \leq A \int_{-\infty}^{\infty} u(t)y(t)\,dt,				\label{ineq_lem1_con}
		\end{equation}
where $y=\phi u$.
\end{lemma}
\begin{proof}
Since $\underline{\alpha}$ is monotone and bounded,  and since $\underline{\beta}$ is monotone, bounded and odd, it follows (e.g. \cite{desoer75}, p205) that for any $u\in\mathcal{L}_2$,
\begin{equation}
	\int_{-\infty}^{\infty}
				u(t+\tau)\underline{\alpha}(u(t)) 
			\,dt
\leq
		\int_{-\infty}^{\infty}
				u(t)\underline{\alpha}(u(t)) 
			\,dt, \label{ineq03}
\end{equation}
and
\begin{equation}
	\left |
	\int_{-\infty}^{\infty}
				u(t+\tau)\underline{\beta}(u(t)) 
			\,dt
	\right |
\leq
		\int_{-\infty}^{\infty}
				u(t)\underline{\beta}(u(t)) 
			\,dt.\label{ineq04}
\end{equation}

Let $u=u_+-u_-$ and $y=y_+-y_-$ be the Jordon measure decompositions of $u$ and $y$ respectively. Then for any $t,\tau\in\mathbb{R}$,
\begin{eqnarray}
u(t+\tau)y(t) & = & \left [u_+(t+\tau)-u_-(t+\tau)\right ] \left [y_+(t)-y_-(t)\right ],\nonumber\\
& \leq &  u_+(t+\tau)y_+(t)+u_-(t+\tau)y_-(t),\nonumber\\
& \leq & u_+(t+\tau)\overline{\alpha}(u_+(t))+u_-(t+\tau)\overline{\alpha}(u_-(t)),\nonumber\\
& &\mbox{by Assumption~\ref{ass1},}\nonumber\\
& \leq & Au_+(t+\tau)\underline{\alpha}(u_+(t))+Au_-(t+\tau)\underline{\alpha}(u_-(t)),\nonumber\\
& &\mbox{by Assumption~\ref{ass1a}.}\label{ineq05}
\end{eqnarray}
Hence
\begin{eqnarray}
\lefteqn{
		\int_{-\infty}^{\infty}  u(t+\tau)y(t) \,dt 
	}\nonumber\\
	& \leq &  A\int_{-\infty}^{\infty}\left [u_+(t)\underline{\alpha}(u_+(t))+u_-(t)\underline{\alpha}(u_-(t))\right ]\,dt
\mbox{ by (\ref{ineq03}),}\nonumber \\
	& = & A \int_{-\infty}^{\infty}u(t)\underline{\alpha}(u(t))\,dt,\nonumber\\
	& \leq & A \int_{-\infty}^{\infty}  u(t)y(t) \,dt\mbox{ by Assumption~\ref{ass1}.} 
\end{eqnarray}

Furthermore, for any $t,\tau\in\mathbb{R}$,
\begin{eqnarray}
|u(t+\tau)y(t)| & = & |u(t+\tau)|\,|y(t)|,\nonumber\\
& \leq & |u(t+\tau)|\,|\overline{\beta}(u(t))|\mbox{ by Assumption~\ref{ass1}},\nonumber\\
& \leq & B|u(t+\tau)|\,|\underline{\beta}(u(t))|\mbox{ by Assumption~\ref{ass1a}},\nonumber\\
& = &  B|u(t+\tau)|\,\underline{\beta}(|u(t)|)\mbox{ since $\underline{\beta}$ is odd.}\label{B_ineq}
\end{eqnarray}
Hence
\begin{eqnarray}
\lefteqn{
	\left |
		\int_{-\infty}^{\infty}  u(t+\tau)y(t) \,dt
	\right |
	\leq 
		\int_{-\infty}^{\infty}
			\left |
				u(t+\tau)
				y(t)
			\right |
			\,dt,}\nonumber\\
	& \leq &
		B\int_{-\infty}^{\infty}
			\left |
				u(t+\tau)
			\right |\,
				\underline{\beta}(|u(t)|)\mbox{ by (\ref{B_ineq}),} \nonumber\\
	& \leq & B
		\int_{-\infty}^{\infty}
			\left |
				u(t)
			\right |\,
				\underline{\beta}(|u(t)|)
			\,dt 
				\mbox{ by (\ref{ineq04}),}\nonumber\\
	& = & B
		\int_{-\infty}^{\infty}
				u(t)
				\underline{\beta}(u(t)) 
			\,dt\mbox{ since $\underline{\beta}$ is odd,}\nonumber\\
	& \leq & B
		\int_{-\infty}^{\infty}
				u(t)
				y(t)
			\,dt\mbox{ by Assumption~\ref{ass1}. } \qedhere
\end{eqnarray}
\end{proof}

Define $\mathbf{H}$ as the set of generalized functions $h(\cdot)$ of the form
\begin{equation}
h(t) = \sum_i h_i \delta(t-t_i)+h_a(t),
\end{equation}
with $t_i\neq 0$, $h_a(0)=0$, $h_i\in\mathbb{R}$ for all $i$ and $h_a(t)\in\mathbb{R}$ for all $t\in\mathbb{R}$. In addition, define the norm 
(c.f. \cite{desoer75,Vidyasagar78}\footnote{
						The notations $\|\cdot\|_{\mathcal{A}}$ and $\|\cdot\|_{\mathbf{A}}$ are used in \cite{desoer75} and \cite{Vidyasagar78} respectively.
						}):
\begin{equation}
	\|h\|_H \triangleq \sum_i |h_i| + \int_{-\infty}^\infty |h_a(t)|\,dt<\infty.
\end{equation}
Define $\mathbf{H_p}$ as the subset of $\mathbf{H}$ where $h_i\geq 0$ for all $i$ and $h_a(t)\geq 0$ for all $t\in\mathbb{R}$. 
We establish the following generalization of the Zames-Falb theorem.

\begin{theorem}[quasi-monotone or quasi-odd nonlinearity]\label{thm1} Under the conditions of Assumptions~\ref{ass0},~\ref{ass1}, and~\ref{ass1a}, let $H_+$ and $H_-$ be noncausal convolution operators whose respective impulse responses are $h_+\in\mathbf{H_p}$ and $h_-\in\mathbf{H_p}$ satisfying
	\begin{equation}
	A\|h_+\|_H+B\|h_-\|_H < 1.\label{h_ineq01}
	\end{equation}
Let  
$M=1-H_++H_-$. 
Then for any   $u\in\mathcal{L}_{2}$, 
\begin{equation}
	\int_{-\infty}^\infty  (M u)(t)\, (\phi u)(t) \, dt \geq 0.\label{int_thm1}
\end{equation}
Furthermore the continuous-time Lurye system (\ref{eq:Lurye}) is stable provided there exists $\varepsilon > 0$ such that
\begin{equation}
		\mbox{Re}\left [
				M(j\omega) G(j\omega)
				\right ] \geq \varepsilon\mbox{ for all }\omega\in\mathbb{R}.\label{ineq_thm1}
\end{equation}
\end{theorem}
\begin{proof}
Let $y=\phi u$ and let $m$ be the impulse response\footnote{There is a typo in \cite{desoer75} that we repeat throughout \cite{HeathCDC15}. The impulse response of $M$ is $m(t)=\delta(t)-h_+(t)+h_-(t)$.} of $M$. Then
\begin{eqnarray}
\lefteqn{
\int_{-\infty}^{\infty}(M u)(t)\,(\phi u)(t)  \,dt =\int_{-\infty}^{\infty}(m*u)(t)y(t)  \,dt}\nonumber\\
	& = &  \int_{-\infty}^{\infty} u(t) y(t)  \,dt -  \int_{-\infty}^{\infty} (h_+*u)(t) y(t)\,dt \nonumber\\
& & +  \int_{-\infty}^{\infty} (h_-*u)(t) y(t) \, dt\nonumber\\
& \geq & (1 - A\|h_+\|_H-B\|h_-\|_H)\int_{-\infty}^\infty u(t) y(t) \,dt\nonumber\\
	& &\mbox{ by Lemma~\ref{lem1},}\nonumber\\
& \geq & 0\mbox{ provided (\ref{h_ineq01}) holds. } 
\end{eqnarray}
Since $M$ belongs to a subclass of the Zames-Falb multipliers, $M^{-1}$ exists. This establishes the positivity of the map from $x$ to $y$ where $u=M^{-1} x$ (see Fig~\ref{fig_3}). Similarly an appropriate  factorization of $M$ is guaranteed (see also \cite{Carrasco:12}). Stability then follows from standard multiplier theory (see e.g. \cite{desoer75}).
\begin{figure}[t]
	\begin{center}
  \ifx\JPicScale\undefined\def\JPicScale{1}\fi
\unitlength \JPicScale mm
\begin{picture}(55,10)(0,0)
\linethickness{0.3mm}
\put(10,10){\line(1,0){10}}
\put(10,0){\line(0,1){10}}
\put(20,0){\line(0,1){10}}
\put(10,0){\line(1,0){10}}
\linethickness{0.3mm}
\put(35,10){\line(1,0){10}}
\put(35,0){\line(0,1){10}}
\put(45,0){\line(0,1){10}}
\put(35,0){\line(1,0){10}}
\linethickness{0.3mm}
\put(0,5){\line(1,0){10}}
\put(0,5){\vector(-1,0){0.12}}
\linethickness{0.3mm}
\put(20,5){\line(1,0){15}}
\put(20,5){\vector(-1,0){0.12}}
\linethickness{0.3mm}
\put(45,5){\line(1,0){10}}
\put(45,5){\vector(-1,0){0.12}}
\put(15,5){\makebox(0,0)[cc]{$\phi$}}

\put(40,5){\makebox(0,0)[cc]{$M^{-1}$}}

\put(50,10){\makebox(0,0)[cc]{$x$}}

\put(27.5,10){\makebox(0,0)[cc]{$u$}}

\put(5,10){\makebox(0,0)[cc]{$y$}}

\end{picture}
	\caption{The proof of Theorem~\ref{thm1} establishes positivity from $x$ to $y$ where $x(t)=(m*u)(t)$. Note that $M^{-1}$ exists because $M$ belongs to a subclass of the Zames-Falb multipliers.}\label{fig_3}
	\end{center}
\end{figure}
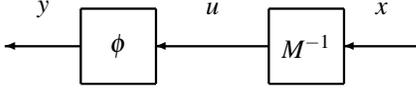
\end{proof}



\begin{remark}\label{remark1}
The positivity result of Theorem~\ref{thm1} is sufficient to establish stability using classical theory \cite{desoer75}. Nevertheless it is straightforward to establish stability via the integral quadratic constraint (IQC) theory of \cite{megretski97}. 
Specifically it follows immediately from Lemma~\ref{lem1} that  
\begin{equation}
	\tau \int_{-\infty}^\infty  (M u)(t)\, (\phi u)(t) \, dt \geq 0\mbox{ for any }\tau\geq 0.
\end{equation}
Thus the homotopy argument of \cite{megretski97} can be used to establish stability. 

Similarly the stability result of Theorem~\ref{thm1} can be established  via the theory of delay-integral-quadratic constraints \cite{Yakubovich:02,Altshuller:04,Altshuller:11,Altshuller:13}. Specifically Lemma~\ref{lem1} establishes the time-domain quadratic forms used to give the frequency domain stability criterion of Theorem~\ref{thm1}. The relation between the IQC theory of \cite{megretski97} and delay-integral-quadratic constraints is discussed in \cite{Carrasco15}.
\end{remark}

\begin{remark}
We can write $M=1-H$ where $H$ is a noncausal convolution operator whose impulse response is $h\in\mathbf{H}$. The Jordan decomposition of $h$ is $h=h_+-h_-$.
\end{remark}

\section{Special cases}\label{sec_mon}

\subsection{Quasi-odd nonlinearities}

If the nonlinearity is time-invariant, bounded and monotone  we can set $\overline{\alpha}=\underline{\alpha}$ so $A=1$. The Zames-Falb theorem for such nonlinearities follows  immediately  by setting $h_-=0$.
The Zames-Falb theorem for monotone, bounded and odd nonlinearities also follows immediately  when $A=B=1$.

The Zames-Falb theorem for odd nonlinearities allows a much wider class of multipliers, but if the nonlinearity is quasi-odd  then it cannot be used. Yet the Zames-Falb theorem for more general nonlinearities  is independent of the value $B$. This immediately suggests an intermediate result.
\begin{corollary}[quasi-odd nonlinearity]\label{cor3}

Under the conditions of Assumptions~\ref{ass0},~\ref{ass1} and~\ref{ass1a} with $A=1$, let $H_+$ and $H_-$ be noncausal convolution operators whose impulse responses are respectively $h_+\in\mathbf{H_p}$  and $h_-\in\mathbf{H_p}$ satisfying
	\begin{equation}
	\|h_+\|_H+ B\|h_-\|_H< 1. \label{ineq_cor3}
	\end{equation}
Let  $M=1-H_++H_-$. 
Then the positivity condition (\ref{int_thm1})  holds and the 
  Lurye system of Fig 1 is stable provided (\ref{ineq_thm1})  holds.
\end{corollary}
\begin{proof}
Immediate from Theorem~\ref{thm1} when $A=1$. 
\end{proof}

\subsection{Quasi-monotone nonlinearities with odd bounds}

Both  \cite{Rantzer2001} and \cite{Materassi11} consider single-valued non-monotone nonlinearities with odd bounds.  In our terminology  $\underline{\alpha}=\underline{\beta}$, $\overline{\alpha}=\overline{\beta}$ and $A=B>1$.


In \cite{Rantzer2001} a time-invariant stiction nonlinearity is given as $\alpha(u(t))=ku(t)/\varepsilon$ when $|u(t)|$ is small, and $1\leq |\alpha(u(t))|\leq 1+\delta$ when $|u(t)|$ is large. This nonlinearity can be bounded 
below by $\underline{\beta}(u(t)) = \mbox{sign}(u(t)) \times \min(k |u(t)|/\varepsilon,1)$ and above by 
$B\underline{\beta}$ with   $B=1+\delta$. In  \cite{Rantzer2001} a stability condition is given equivalent to choosing $M=1-H$ where $H$ has impulse response $h\in\mathbf{H}$ and $(1+\delta)\|h\|_H < 1$.

In \cite{Materassi11} a nonlinearity  is bounded below by  $(1-D)\hat{\beta}$ and above by $(1+D)\hat{\beta}$  for some monotone, bounded and odd ``skeleton'' $\hat{\beta}:\mathbb{R}\rightarrow\mathbb{R}$. Hence $B=(1+D)/(1-D)$.  In  \cite{Materassi11} a stability condition is given equivalent to choosing $M=1-H$ where $H$ has impulse response $h\in\mathbf{H}$ and $(1+D)^2/(1-D)^2\|h\|_H < 1$. It is observed in \cite{Altshuller:13} that it is sufficient to require $(1+D)/(1-D)\|h\|_H < 1$. The result in \cite{Altshuller:13} also extends to time-periodic (but not more general) nonlinearities.

Both the results of \cite{Rantzer2001} and \cite{Materassi11}, as well as the latter's refinement in \cite{Altshuller:13}, can be expressed as a corollary of Theorem~\ref{thm1} with $A=B$. The further generalisation that the nonlinearity need be neither memoryless nor time-invariant. 
\begin{corollary}[quasi-monotone nonlinearity with odd bounds, c.f. \cite{Rantzer2001}, \cite{Materassi11}, \cite{Altshuller:13}]\label{cor4}
	Under the conditions of Assumptions~\ref{ass0},~\ref{ass1} and~\ref{ass1a} with $A=B$,  let $H$ be a noncausal convolution operator whose impulse response is $h\in\mathbf{H}$  satisfying
	\begin{equation}
	B\|h\|_H < 1.\label{ineq_cor4}
	\end{equation}
Let  $M=1-H$. 
Then the positivity condition (\ref{int_thm1})  holds and the  Lurye system of Fig 1 is stable provided (\ref{ineq_thm1}) holds.
\end{corollary}
\begin{proof}
Setting $h=h_+-h_-$ to be the Jordan decomposition of $h$ with $h_+\in\mathbf{H_p}$ and $h_-\in\mathbf{H_p}$, together with  (\ref{ineq_cor4}) are sufficient for (\ref{h_ineq01}) in Theorem~\ref{thm1}. 
\end{proof}



\section{Loop transformation}\label{sec_loop}


In classical multiplier analysis \cite{Zames68,desoer75} it is standard to apply loop transformations (Fig~\ref{loop_t}) when the nonlinearity is slope-restricted. Similarly we may apply loop transformations when a nonlinearity has slope-restricted bounds.
\begin{lemma}\label{lem3}
Under the conditions of Assumption~\ref{ass1}, suppose $\underline{\alpha}$ and  $\overline{\alpha}$ are both, in addition, slope-restricted on $[0,s]$. Let $k>s$.  Define $\underline{\alpha}_k$ as the map from $u(t)-\underline{\alpha}(u(t))/k$ to $\underline{\alpha}(u(t))$ and define $\overline{\alpha}_k$ similarly. Then $\underline{\alpha}_k$ and $\overline{\alpha}_k$ are both monotone and bounded. Furthermore the map from $u(t)-y(t)/k$ to $y(t)$ is bounded below by  $\underline{\alpha}_k$ and above by $\overline{\alpha}_k$. 

A similar statement follows if we define $\underline{\beta}_k$ and $\overline{\beta}_k$ similarly. In addtion, $\underline{\beta}_k$ and $\overline{\beta}_k$ are both odd.

\end{lemma}
\begin{proof}  It is well-known that $\underline{\alpha}_k$ and $\overline{\alpha}_k$ are  monotone and bounded, and that $\underline{\beta}_k$ and $\overline{\beta}_k$ are monotone, bounded and odd \cite{Zames68,desoer75}.

Suppose (without loss of generality) that $u(t)>0$. Then
$0\leq \underline{\alpha}(u(t))\leq y(t) \leq \overline{\alpha}(u(t)).$ Similarly $u(t)-\underline{\alpha}(u(t))/k\geq u(t)-y(t)/k\geq u(t)-\overline{\alpha}(u(t))/k \geq 0$. Hence
\begin{equation}
0\leq
\frac{\underline{\alpha}(u(t))}{u(t)-\underline{\alpha}(u(t))/k}
\leq 
\frac{y(t)}{u(t)-y(t)/k}
\leq
\frac{ \overline{\alpha}(u(t))}{u(t)-\overline{\alpha}(u(t))/k}.
\end{equation}
The result for $\underline{\beta}_k$ and $\overline{\beta}_k$ follows similarly.
\end{proof}
\begin{remark}
Even when the mapping from $u(t)$ to $y(t)$ is single-valued, the mapping from $u(t)-y(t)/k$ to $y(t)$ need not be \cite{desoer75}. In particular the slope of the mapping from $u(t)$ to $y(t)$ may exceed $k$. The results of \cite{Rantzer2001}, \cite{Materassi11} cannot be applied with loop transformations without either further restrictions or the generalisation in Theorem~\ref{thm1} to possibly multivalued nonlinearities.
\end{remark}
We require a counterpart to Assumption~\ref{ass1a}.
\begin{assumption}\label{ass2}
Given Assumption~\ref{ass1}, suppose in addition $\underline{\alpha}$ and $\overline{\alpha}$ are both slope-restricted on $[0,s]$. Let $\underline{\alpha}_k$ and $\overline{\alpha}_k$ be defined as in Lemma~\ref{lem3} with $k>s$. Then there is assumed to be some finite $A_k\geq 1$ such that  $\overline{\alpha}_k$ is bounded above by $A_k\underline{\alpha}_k$.

 Similarly suppose in addition $\underline{\beta}$ and $\overline{\beta}$ are also both slope restricted on $[0,s]$. Let $\underline{\beta}_k$ and $\overline{\beta}_k$ also be defined as in Lemma~\ref{lem3} with $k>s$. Then there is assumed to be some (possibly infinite) $B_k\geq A_k$ such that  $\overline{\beta}_k$ is bounded above by $B_k\underline{\beta}_k$.
\end{assumption}

\begin{lemma}\label{lem4}
Under the conditions of Assumptions~\ref{ass0},~\ref{ass1} and~\ref{ass2},
	 if $\phi:\mathcal{L}_{2e}\rightarrow\mathcal{L}_{2e}$ then for all $u\in\mathcal{L}_2$ and for all $\tau\in\mathds{R}$, 
		\begin{multline}
			-B_k \int_{-\infty}^{\infty}\left (u(t)-y(t)/k\right ) y(t)\,dt \leq \\
			\int_{-\infty}^{\infty}\left (u(t+\tau)-y(t+\tau)/k\right )  y(t)\,dt \leq\\
 			A_k \int_{-\infty}^{\infty}\left (u(t)-y(t)/k\right )y(t)\,dt,
		\end{multline}
where $y=\phi u$.
\end{lemma}
\begin{proof}
Immediate from Lemmas~\ref{lem1} and~\ref{lem3} and Assumption~\ref{ass2}.
\end{proof}

\begin{theorem}[quasi-monotone or quasi-odd nonlinearity with slope-restricted bounds]\label{thm2} 

Under the conditions of Assumptions~\ref{ass0},~\ref{ass1} and~\ref{ass2}, let $H_+$ and $H_-$ be noncausal convolution operators whose respective impulse responses are $h_+\in\mathbf{H_p}$ and $h_-\in\mathbf{H_p}$  satisfying
	\begin{equation}
	A_k\|h_+\|_H+B_k\|h_-\|_H   < 1.
		\label{h_ineq01a}
	\end{equation}
Let  $M=1-H_++H_-$. 
Then for any   $u\in\mathcal{L}_{2}$ 
\begin{equation}
	\int_{-\infty}^\infty  M(u-\phi u/k)(t)\, (\phi u)(t) \, dt \geq 0,\label{int_thm2}
\end{equation} 
and the  Lurye system of Fig 1 is stable provided there exists $\varepsilon>0$ such that
\begin{equation}
		\mbox{Re}\left [
			M(j\omega) (1+kG(j\omega))
		\right ] \geq \varepsilon\mbox{ for all }\omega\in\mathbb{R}.\label{ineq_thm2}
\end{equation}
\end{theorem}
\begin{proof}
Similar to that of Theorem~\ref{thm1}.
\end{proof}

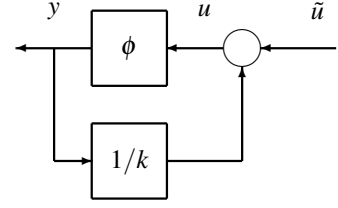
\begin{figure}[t]
	\begin{center}
	\ifx\JPicScale\undefined\def\JPicScale{1}\fi
\unitlength \JPicScale mm
\begin{picture}(42.5,25)(0,0)
\linethickness{0.3mm}
\put(10,25){\line(1,0){10}}
\put(10,15){\line(0,1){10}}
\put(20,15){\line(0,1){10}}
\put(10,15){\line(1,0){10}}
\linethickness{0.3mm}
\put(10,10){\line(1,0){10}}
\put(10,0){\line(0,1){10}}
\put(20,0){\line(0,1){10}}
\put(10,0){\line(1,0){10}}
\linethickness{0.3mm}
\put(0,20){\line(1,0){10}}
\put(0,20){\vector(-1,0){0.12}}
\linethickness{0.3mm}
\put(20,20){\line(1,0){7.5}}
\put(20,20){\vector(-1,0){0.12}}
\linethickness{0.3mm}
\put(32.5,20){\line(1,0){10}}
\put(32.5,20){\vector(-1,0){0.12}}
\put(15,20){\makebox(0,0)[cc]{$\phi$}}

\put(15,5){\makebox(0,0)[cc]{$1/k$}}

\put(40,25){\makebox(0,0)[cc]{$\tilde{u}$}}

\put(25,25){\makebox(0,0)[cc]{$u$}}

\put(5,25){\makebox(0,0)[cc]{$y$}}

\linethickness{0.3mm}
\put(30,20){\circle{5}}

\linethickness{0.3mm}
\put(5,5){\line(0,1){15}}
\linethickness{0.3mm}
\put(5,5){\line(1,0){5}}
\put(10,5){\vector(1,0){0.12}}
\linethickness{0.3mm}
\put(20,5){\line(1,0){10}}
\linethickness{0.3mm}
\put(30,5){\line(0,1){12.5}}
\put(30,17.5){\vector(0,1){0.12}}
\end{picture}
	\caption{Loop transformation. The map from $\tilde{u}(t)=u(t)-y(t)/k$ to $y(t)$ may be multivalued, even if the map from $u(t)$ to $y(t)$ is single-valued. Similarly the ratios of bounds on the nonlinearity (i.e. $A$ and $B$ for the map from $u(t)$ to $y(t)$ and $A_k$ and $B_k$ for the map from $u(t)-y(t)/k$ to $y(t)$) are not necessarily preserved under loop transformation.}\label{loop_t}
	\end{center}
\end{figure}

There is no guarantee that $A_k$ (or $B_k$) is small, even when $A$ (or $B$) is. In fact it is straightforward to construct examples where $A_k$ (or $B_k$) can be arbitrarily large. By the same token, there are cases where $A_k=A$ (and where $B_k=B$).  Here we consider such a case for $A_k$ where the nonlinearity is monotone.

\begin{lemma}\label{lem5}
Suppose the nonlinearity is monotone. Then $A_k=A=1$.
\end{lemma}
\begin{proof}
We may set $y(t)=\alpha u(t)$ for some monotone $\alpha$ and $\overline{\alpha}=\underline{\alpha}=\alpha$. Hence  $\overline{\alpha}_k = \underline{\alpha}_k$.
\end{proof}

\begin{corollary}[quasi-odd nonlinearity with slope-restricted bounds]\label{cor5}
Under the conditions of Theorem~\ref{thm2}, suppose the nonlinearity  is in addition monotone. Then (\ref{h_ineq01a}) may be replaced by the condition
	\begin{equation}
	\|h_+\|_H+B_k\|h_-\|_H  < 1.
	\end{equation}
\end{corollary}
\begin{proof}
Immediate from Theorem~\ref{thm2} and Lemma~\ref{lem5}.
\end{proof}

\section{Example with deadzone and monotone slope-restricted bounds}\label{sec:ex01}
In this section and the next we illustrate the practical applicability of the multipliers. The example in this section is continuous-time while the example in the next is discrete-time. In both cases we exploit loop transformation.

In this section we give an example of a class of nonlinearity with deadzone where Theorem~\ref{thm2} gives better results than the circle criterion. It is similar in spirit to an example in \cite{Materassi11} but differs in that the deadzone need not be symmetric, the bounds need not be symmetric, the nonlinearity itself need be neitehr memoryless nor time-invariant and we apply loop transformation. The example illustrates how the values $A_k$ and $B_k$ may differ from $A$ and $B$ and hence the set of multipliers available if we apply Theorem~\ref{thm2} may be smaller than the set available if we apply Theorem~\ref{thm1}.

\subsection{Nonlinearity with deadzone and monotone slope-restricted bounds}

Suppose the nonlinearity is bounded by
\begin{equation}
\underline{\alpha}(u(t)) = 
	\left \{
		\begin{array}{lcc}
			s_{n1} (u(t)+d_n)  &\mbox{for}&u(t)<-d_n\\
			0 & \mbox{for}& -d_n\leq u(t) \leq d_p\\
			s_{p1} (u(t)-d_p) & \mbox{for}&u(t)>d_p
		\end{array}
	\right .
\end{equation}
and
\begin{equation}
\overline{\alpha}(u(t)) = 
	\left \{
		\begin{array}{lcc}
			s_{n2} (u(t)+d_n)  &\mbox{for}&u(t)<-d_n\\
			0 & \mbox{for}& -d_n\leq u(t) \leq d_p\\
			s_{p2} (u(t)-d_p) & \mbox{for}&u(t)>d_p
		\end{array}
	\right .
\end{equation}
with $0<s_{n1}< s_{n2}$, $0<s_{p1}<s_{p2}$ and $d_n>0$, $d_p>0$ (Fig~\ref{Fig_ex1}).
\begin{figure}[t]
\begin{center}
\includegraphics[width=0.8\linewidth]{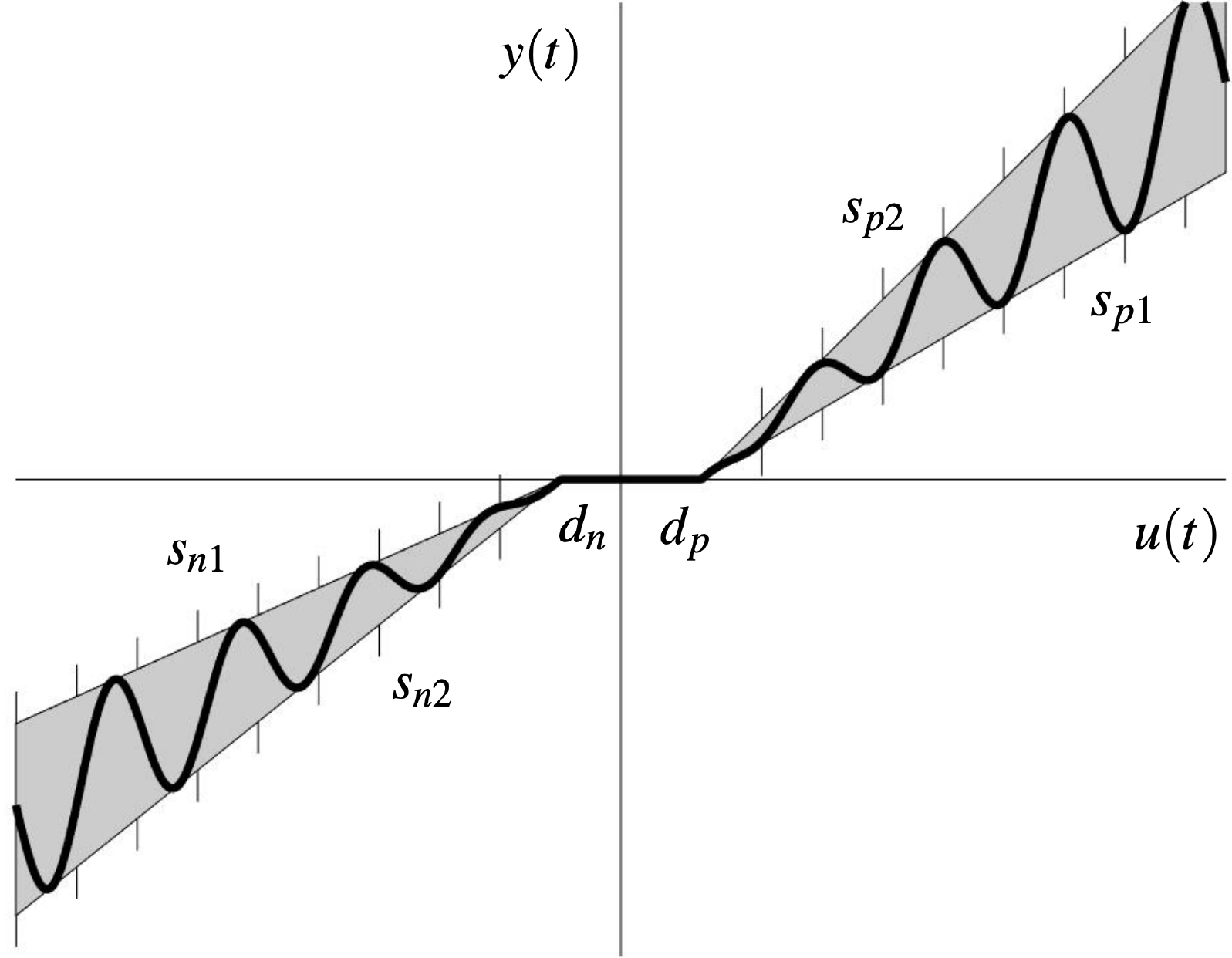}
\caption{Nonlinearity bounds for the example of Section~\ref{sec:ex01}. The analysis of \cite{Rantzer2001}, \cite{Materassi11} and  \cite{Altshuller:13}  cannot be used when $d_n\neq d_p$. We provide a specific example where Theorem~\ref{thm2} gives better results than the circle criterion.}\label{Fig_ex1}
\end{center}
\end{figure}
 It follows that
\begin{equation}
A = \max\left (\frac{s_{n2}}{s_{n1}},\frac{s_{p2}}{s_{p1}}\right ),
\end{equation}
and
\begin{equation}
B = \frac
		{
			\max(s_{n2},s_{p2})
		}
		{
			\min(s_{n1},s_{p1})
		} 
	\mbox{ when } d_n=d_p,
\end{equation}
but there is no finite $B$ when $d_n\neq d_p$.

Both $\underline{\alpha}$ and $\overline{\alpha}$ are monotone and slope-restricted on $[0,s]$ with $s=\max(s_{n2},s_{p2})$. If we apply a loop transformation with $k>s$ the new bounds $\underline{\alpha}_k$ and $\overline{\alpha}_k$ of Lemma~\ref{lem3} are given by
\begin{equation}
\underline{\alpha}_k(u(t)) = 
	\left \{
		\begin{array}{lcc}
			\tilde{s}_{n1} (u(t)+d_n)  &\mbox{for}&u(t)<-d_n\\
			0 & \mbox{for}& -d_n\leq u(t) \leq d_p\\
			\tilde{s}_{p1} (u(t)-d_p) & \mbox{for}&u(t)>d_p
		\end{array}
	\right .
\end{equation}
and
\begin{equation}
\overline{\alpha}_k(u(t)) = 
	\left \{
		\begin{array}{lcc}
			\tilde{s}_{n2} (u(t)+d_n)  &\mbox{for}&u(t)<-d_n\\
			0 & \mbox{for}& -d_n\leq u(t) \leq d_p\\
			\tilde{s}_{p2} (u(t)-d_p) & \mbox{for}&u(t)>d_p
		\end{array}
	\right .
\end{equation}
with 
\begin{equation}
\begin{aligned}
\tilde{s}_{n1}=\frac{ks_{n1}}{k-s_{n1}}
\mbox{, }
\tilde{s}_{n2}=\frac{ks_{n2}}{k-s_{n2}}
\mbox{, }\\
\tilde{s}_{p1}=\frac{ks_{p1}}{k-s_{p1}}
\mbox{ and }
\tilde{s}_{p2}=\frac{ks_{p2}}{k-s_{p2}}.
\end{aligned}
\end{equation}
Hence
\begin{equation}
	A_k = \max \left ( \frac{s_{n2}}{s_{n1}}\frac{k-s_{n1}}{k-s_{n2}}, \frac{s_{p2}}{s_{p1}}\frac{k-s_{p1}}{k-s_{p2}}\right )
	\end{equation}
and $B_k$ can be found similarly when $d_n=d_p$.

\subsection{Stability criteria}

Now consider the continuous-time Lurye system (\ref{eq:Lurye}). The circle criterion can be used to establish stability provided $\mbox{Re}[G(j\omega)] > -1/\max(s_{n2},s_{p2})$ for all $\omega$. If the nonlinearity itself is time-varying the Popov criterion cannot be used, and similalry if the nonlinearity is not monotone then the Zames-Falb criterion cannot be used.  If $d_n\neq d_p$ then there is no finite $B$  so none of the criteria of \cite{Rantzer2001}, \cite{Materassi11} or  \cite{Altshuller:13} can be used to establish stability. However either Theorem~\ref{thm1} or Theorem~\ref{thm2} may be used. Here we illustrate the use of Theorem~\ref{thm2}.

\subsection{Specific example}

As a specifc example, let $G$ be the resonant system with delay
\begin{equation}\label{G_ex1}
	G(s) = e^{-s/5} \frac{1}{s^2+0.3s+1}
\end{equation}
and let $s_{n1}=s_{p1} = 0.5$ and $s_{n2}=s_{p2}=0.6$.

For this example the circle criterion fails to establish stability since  $\min\left (\mbox{Re }[G(j\omega)]\right )\approx -1.8195 < -1/0.6$. Similarly Theorem 1 fails to stablish stability directly since the phase of $G$ drops to $-\infty$ as $\omega\rightarrow\infty$. 

However, if we apply a loop transformation with $k=1$ we obtain
$\tilde{s}_{n1}=\tilde{s}_{p1}=1$, $\tilde{s}_{n2}=\tilde{s}_{p2}=3/2$ and hence $\tilde{R}_m = 3/2$. Define the multiplier
	$M_\varepsilon(s) = 1 - (2/3-\varepsilon) e^{-0.7s}\mbox{ with }\varepsilon\geq0$.
We find 
	$-90^o < \angle M_0(j\omega) (1+G(j\omega)) < 90^o$ 
for all $\omega$.
If follows by continuity that there is an $\varepsilon>0$ such that  Theorem~\ref{thm2} with multiplier $M_\varepsilon$  establishes stability.

\section{Example  with asymmetric saturation}\label{sec_sat}\label{sec:ex02}
\subsection{Asymmetric saturation}

One of our motivations is to study asymmetric saturation. 
This is of high practical importance as it corresponds to the case with odd bounds but constant offset (due to non-zero setpoint or disturbance). 
It
 is possible for a Lurye system to be stable with symmetric saturation but unstable with asymmetric saturation \cite{HeathECC15,Heath14} and 
the behaviour of Lurye systems with asymmetric saturation continues to be of interest \cite{Groff:19}.  
In this example stability is guaranteed with exogenous signals whose steady-state is small but exhibits cycling with exogenous signals whose steady state is large.

Define the asymmetric saturation with gain~$s$ as:
\begin{equation}
\mbox{sat}_{s,-m,n}(u(t)) = \left \{
\begin{array}{l}
-m\mbox{ for } u(t) < \frac{-m}{s},\\
su(t)\mbox{ for } \frac{-m}{s}\leq u(t) \leq \frac{n}{s}\\
n\mbox{ for } \frac{n}{s} < u(t)
\end{array}
\right .
\label{asat}
\end{equation} 
where $s>0$, $m>0$ and $n>0$  (Fig~\ref{fig_sat}).

The nonlinearity is monotone so $A=1$. Define 
\begin{equation}
	\underline{B}=\min\{m,n\} \mbox{ and }  \overline{B}=\max\{m,n\}. \label{def_Bs}
\end{equation}
The nonlinearity is bounded below by $\underline{\beta}=\mbox{sat}_{s,-\underline{B},\underline{B}}$ and above by  $\overline{\beta}=\mbox{sat}_{s,-\overline{B},\overline{B}}$. Hence it is quasi-odd with $B=\overline{B}/\underline{B}$.

\begin{corollary}[asymmetric saturation]\label{cor6}
Under the conditions of Theorem~\ref{thm2}, suppose $\phi$ is given by the asymmtetric saturation~(\ref{asat}). Then (\ref{h_ineq01a}) may be replaced by the condition
	\begin{equation}
	\|h_+\|_H+B_k\|h_-\|_H < 1
\mbox{ where }
B_k =\overline{B}/\underline{B},
\label{ineq_cor6}
	\end{equation}
where $\overline{B}$ and $\underline{B}$ are given by (\ref{def_Bs}).
 Furthermore we may allow $k=s$.
\end{corollary}
\begin{proof}
Since $\underline{\beta}$ and $\overline{\beta}$ are slope-restricted on $[0,s]$ we can apply Corollary~\ref{cor5}. 
Let $k=s+\varepsilon$ with $\varepsilon>0$.  
The map from $u(t)-\mbox{sat}_{s,-m,n}(u(t))/k$ to $\mbox{sat}_{s,-m,n}(u(t))$ is $\mbox{sat}_{1/\varepsilon,-m,n}(u(t))$.  Hence $\overline{\beta}_k$ is bounded above by $B_k\underline{\beta}_k$.
A limiting argument for $\varepsilon\rightarrow 0$ can be made following~\cite{Carrasco:13}. 
\end{proof}
\begin{remark}
It is straightforward to show that  (\ref{h_ineq01a}) may be replaced by (\ref{ineq_cor6})  for any saturation function $\mbox{sat}_{s,-\mu,\nu}$ with $s>0$, $\underline{B}\leq \mu\leq \overline{B}$ and $\underline{B}\leq \nu\leq \overline{B}$ where $\overline{B}$ and $\underline{B}$ are given by~(\ref{def_Bs}).
\end{remark}

\begin{figure}[t]
	\begin{center}
		\includegraphics[width = 0.8\linewidth]{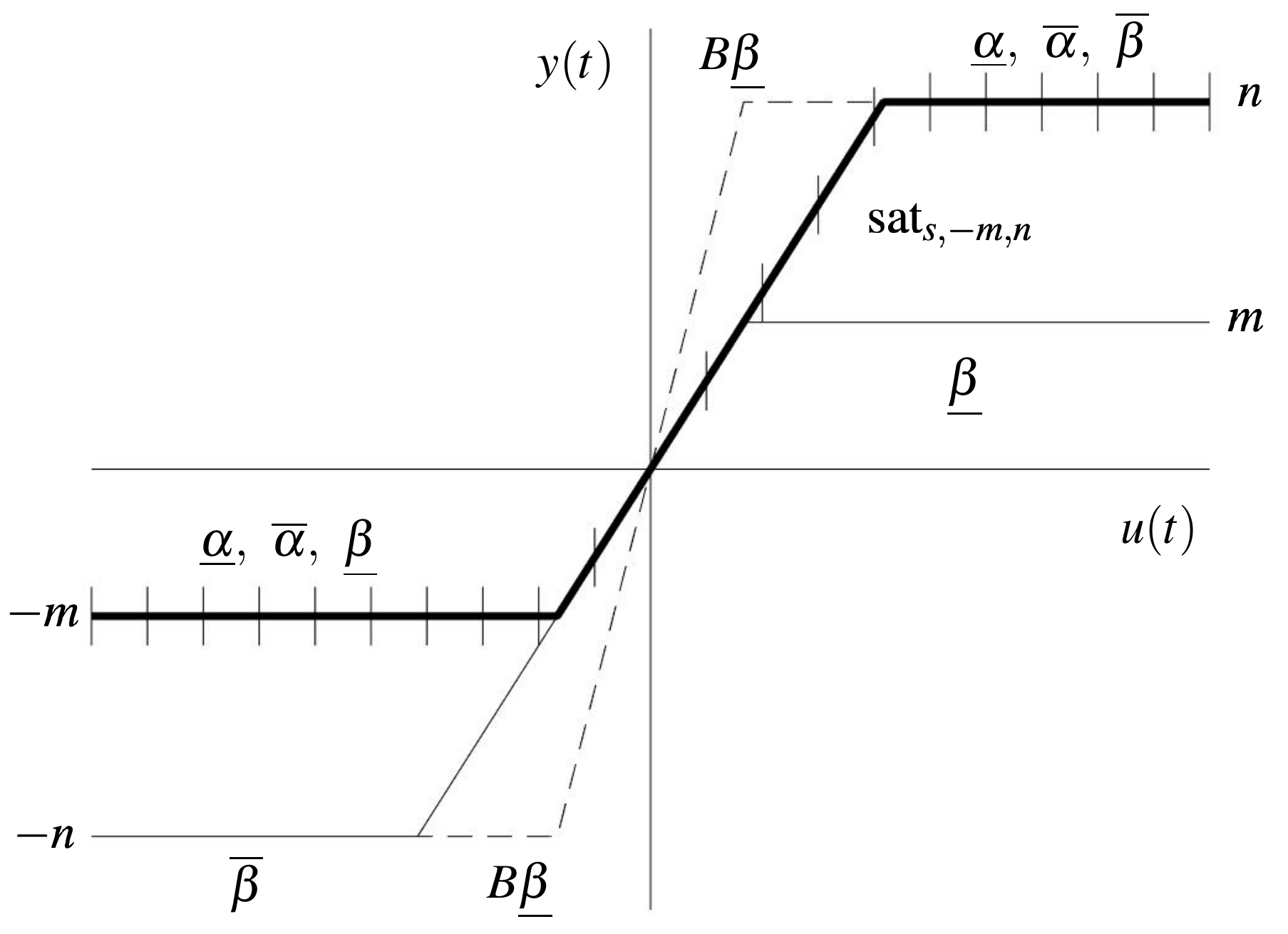}
		\caption{Asymmetric saturation. The nonlinearity is monotone so $\underline{\alpha}=\overline{\alpha}$ and $A=1$. Where $u(t)<0$ we have $\underline{\alpha}(u(t))=\overline{\alpha}(u(t))=\underline{\beta}(u(t))$ and where  $u(t)>0$ we have $\underline{\alpha}(u(t))=\overline{\alpha}(u(t))=\overline{\beta}(u(t))$. Both $\underline{\beta}$ and $\overline{\beta}$ are slope-restricted on $[0,s]$. The bound $\overline{\beta}$ is itself bounded above by $B\underline{\beta}$.
			}\label{fig_sat}
	\end{center}
\begin{center}
\includegraphics[width=\linewidth]{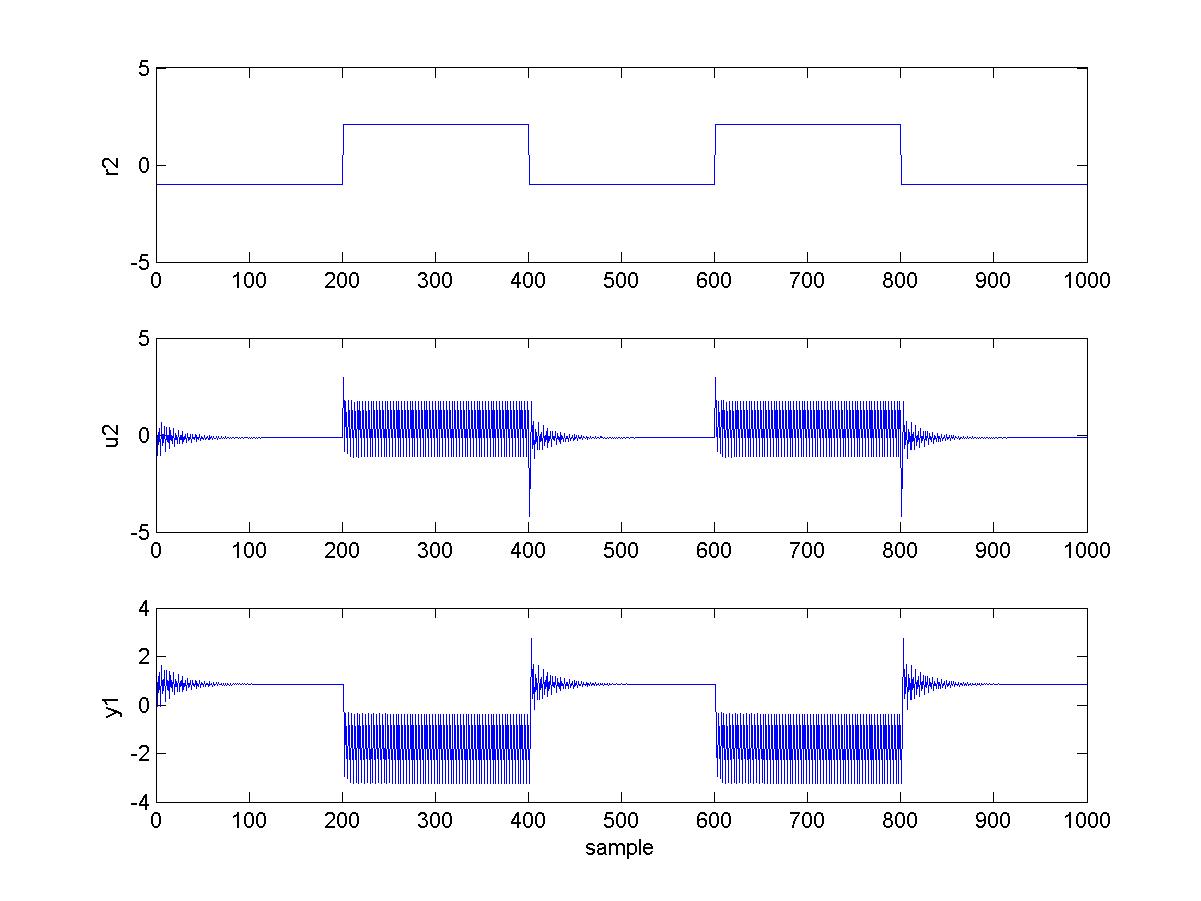}
\caption{Stability changes with size of exogenous signal. For low level, $r_2=-1$ giving steady state $-25/171$  for $u_2$. For high level, $r_2=2.1$ giving steady state $35/114$ for $u_2$. Since the saturation is $\pm 1$ the effective values of $B_k$ are $(1+25/171)/(1-25/171)=98/73\approx 1.3425$ and $(1+35/114)/(1-35/114)=149/79\approx 1.886$.}\label{fig_sim02}
\end{center}
\end{figure}

\subsection{Set-points and disturbances}


Suppose $\phi$ in the Lurye system (\ref{eq:Lurye}) is memoryless and characterised by the symmetric saturation nonlinearity  $\mbox{sat}_{1,-m,m}$ for some $m>0$. It may be of interest to analyse the behaviour when the exogenous signals $r_1$ and/or $r_2$ are step functions and non-zero in steady state. In particular, suppose the system is stable without saturation (i.e. when $\phi$ is replaced by a unit gain) and the signal $u_2$ tends to some $u_s\in\mathbb{R}$ in steady state with $|u_s|<m$. Under what circumstances can we guarantee $u_2$ tends to the same value when there is saturation? Our definition of input-output stability for the Lurye system requires $r_1,r_2\in\mathcal{L}_2$, so it cannot be applied directly in this case. 

But the question is equivalent to asking whether the system is stable when we renormalise our variables so that $r_1$ and $r_2$ are both in $\mathcal{L}_2$ and $u_2$ tends to zero when there is no saturation. In this case the saturation becomes $\mbox{sat}_{1,-m-u_s,m-u_s}$. Hence we can apply Corollary~\ref{cor6} to test for stability with
$B_k=(m+|u_s|)/(m-|u_s|)$.
Observe in particular that the value of $B_k$ is dependent on the magnitudes of the exogenous signals.

\subsection{Specific example}

Here we illustrate the result for asymmetric saturation with a discrete-time example (see Appendix~\ref{sec:discrete}).
Consider the Lurye system (\ref{eq:Lurye}) where $\phi$ is characterised by the saturation function
\begin{equation}
	\alpha(u(t)) = \mbox{sat}_{1,-1,1}(u(t)).
\end{equation}
and let $G$ be the discrete-time transfer function
\begin{equation}
	G(z) = \frac{2z+0.92}{z(z-0.5)}.\label{G_ex2}
\end{equation}
By classical analysis \cite{Willems:68,Willems:71} this is only guaranteed stable when the exogenous signals are zero in steady state.


We find $M(1+G)$ is positive with  $M(z) = 0.596 z+ 1 + 0.022z^{-2} - 0.093z^{-3}$.  Corollary~\ref{cor6} implies  the loop is stable provided $B_k<1.467$. The multipliers were found using a convex search as discussed in Appendix~\ref{sec_search}. Corresponding steady-state values of exogenous signal $r_2$ and input to the saturation~$u_2$  are given in Table~\ref{table}.

 We know \cite{Heath14} there is a three-period limit cycle when $B_k=436/275\approx 1.586$. Fig~\ref{fig_sim02} shows the signals $r_2$, $u_2$ and $y_1$ when $r_2$ is switched between $-1$ and $2.1$ every $200$ samples. Corresponding values of $B_k$ are shown in Table~\ref{table}. The loop is guaranteed stable when $r_2=-1$, but shows a three-period limit cycle when $r_2=2.1$. 

\begin{table}[t]
\begin{center}
\caption{Values of $B_k$ and corresponding steady state values for $u_2$ and $r_2$ in the example.}\label{table}
\begin{tabular}{|c|c|c|c|l|}
\hline
$B_k$ & $|u_2|$ & $|r_2|$ & Stable & Comment\\
\hline
1 & 0 & 0 & Yes & By classical analysis\\
1.343 & 0.146 & 1 & Yes & Low level in simulation\\
1.467 & 0.189 & 1.295 & Yes &  Corollary~\ref{cor5} \\
\hline
1.586 & 0.227 & 1.55 & No & Three-period limit cycle \cite{Heath14}\\
1.886 & 0.307 & 2.1 & No & High level in simulation \\
\hline
\end{tabular}

\end{center}
\end{table}

\section{Conclusion}
We have provided a generalisation of Zames-Falb muliplier theory for both quasi-monotone nonlinearities and quasi-odd nonlinearities. Both the classical results \cite{Zames68,desoer75} and the generalisations of \cite{Rantzer2001,Materassi11} can be stated as special cases. We have also provided the counterpart results for discrete-time systems in Appendix~\ref{sec:discrete}. The results follow classical multiplier analysis \cite{desoer75} but exploit the Jordan decomposition \cite{Billingsley95} of the impulse response $h=h_+-h_-$ of the operator $H$ where the multiplier is $M=1-H$.

Whereas the generalisations of \cite{Rantzer2001,Materassi11} are focused on non-monotone nonlinearities, we also consider nonlinearities that are monotone and quasi-odd. In this case we provide a result (Corollary~\ref{cor3}) that we illustrate via an  example with asymmetric saturation (Section~\ref{sec:ex02}). Our results may be applied to time-varying and multivalued nonlinearites and hence accommodate loop transformation. Unlike the classical results of  \cite{Zames68}, the set of available multipliers $M$ may be reduced after loop transformation; this is illustrated in the example of Section~\ref{sec:ex01}.

In Appendix~\ref{sec_search} we indicate how modifications of existing search algorithms can provide convex searches for the new class of multipliers. Such a search for discrete-time multipliers is used in the example of Section~\ref{sec:ex02} where multiplier theory can be used to test stability according to the magnitude of exogenous signals in steady state.

\appendices

\section{Discrete-time results}\label{sec:discrete}
%
%
%
%

The discrete-time counterparts of the Zames-Falb multipliers were proposed in \cite{Willems:68,Willems:71}. Applications of the discrete-time Zames–Falb multipliers range from input-constrained model predictive control \cite{Heath:07,Petsagkourakis:20} to first order numerical optimization 
algorithms \cite{Lessard:16,Michalowsky:20}. Although they are defined similarly to the continuous-time Zames-Falb multipliers, their properties are significantly different \cite{Wang:18,Zhang:20}.

 Here, for completeness, we state the discrete-time counterpart of Theorem~\ref{thm1}.  Define  $\mathbf{h_p}$ as the set of sequences in $\ell$  where $h_k\geq 0$ for all $k\in\mathbb{Z}$ and $h_0=0$.

\begin{theorem}[discrete-time, quasi-monotone or quasi-odd nonlinearity]\label{thm1d}

Under the discrete-time counterparts of the conditions of Assumptions~\ref{ass0} and~\ref{ass1}, let $H_+$ and $H_-$ be noncausal convolution operators whose respective impulse responses are $h_+\in\mathbf{h_p}$ and $h_-\in\mathbf{h_p}$ satisfying
	\begin{equation}
	A\|h_+\|_1+B\|h_-\|_1 < 1.\label{h_ineq01_d}
	\end{equation}
Let  $M=1-H_++H_-$. 
Then for any   $u\in\mathcal{\ell}_{2}$ 
\begin{equation}
	\sum_{k =-\infty}^\infty  (M u)_k (\phi  u )_k \geq 0.\label{sum_thm1d}
\end{equation}
Furthermore the  discrete-time Lurye system of Fig 1 is stable provided
\begin{equation}
		\mbox{Re}\left [
				M(e^{j\omega}) G(e^{j\omega})
				\right ] > 0\mbox{ for all }\omega\in[0,2\pi].\label{ineq_thm1d}
\end{equation}
\end{theorem}
\begin{proof}
Similar to Theorem~\ref{thm1}. 
\end{proof}

Discrete-time counterparts to Corollaries~\ref{cor3} and~\ref{cor4} follow straightforwardly as do counterparts to Theorem~\ref{thm2} and Corollaries~\ref{cor5} and~\ref{cor6}.

\section{Convex searches}\label{sec_search}

Our construction relies on the Jordan decomposition of the impulse response $h=h_+-h_-$ with $h_+(t)\geq 0$ and $h_-(t)\geq 0$ for all $t\in\mathbb{R}$ (continuous time) or $[h_+]_k\geq 0$ and $[h_-]_k\geq 0$ for all $k\in\mathbb{Z}$ (discrete time). It follows that any search method for Zames-Falb multipliers (or their discrete equivalents) that exploits the characterisation of the impulse response $h$ as the sum of basis functions  can be easily modified to search for the multipliers of this paper. This is the case with Chen and Wen's LMI search \cite{Chen:95} and the convex FIR search for discrete-time multipliers reported in \cite{Wang:TAC}.

Specifically, suppose a search algorithm constructs an impulse response $h$ as
\begin{equation}
h = \sum_{i=1}^N \lambda_i h_i,
\end{equation}
where each $h_i\in\mathbf{H_p}$  satisfies $h_i(0)=0$, $h_i(t)\geq 0$ for all $t\in\mathbb{R}$ and $\|h_i\|_H =1$ (continuous time) or 
where each $h_i\in\mathbf{h_p}$  satisfies $[h_i]_0=0$, $[h_i]_k\geq 0$ for all $k\in\mathbb{Z}$ and $\|h_i\|_1=1$ (discrete time). Then multipliers for monotone and bounded nonlinearities can be parameterised with the convex constraints
\begin{equation}
\lambda_i\geq 0 \mbox{ for }i=1,\ldots,N\mbox{ and } \sum_{i=1}^N\lambda_i < 1.
\end{equation}
Similarly multipliers for monotone, bounded and odd nonlinearities can be parameterised with the convex constraint
\begin{equation}
 \sum_{i=1}^N|\lambda_i| < 1.
\end{equation}

These can be modified to construct an impulse response as $h=h_+-h_-$ with
\begin{equation}
h_+ = \sum_{i=1}^N \lambda_{i+} h_i\mbox{ and } h_- = \sum_{i=1}^N \lambda_{i-} h_i,
\end{equation}
with each $h_{i}$ defined as before. The appropriate convex constraints are then
\begin{equation}
\lambda_{i+}\geq 0
\mbox{ and }
 \lambda_{i-}\geq 0\mbox{ for }i=1,\ldots,N,
\end{equation}
and
\begin{equation}
A \sum_{i=1}^N\lambda_{i+} + B\sum_{i=1}^N\lambda_{i-} < 1.
\end{equation}

In particular, both the continuous-time search of \cite{Chen:95} and the discrete-time search of \cite{Wang:TAC} may be modified in this way to give LMI-based convex searches. We use such a modified discrete-time search in the  example of Section~\ref{sec:ex02}.



\section*{Dedication}
We dedicate this paper to our late collaborator and co-author Dmitry Altshuller. Had he lived this paper would surely have had a different flavour. We have preserved his spelling of Lurye throughout. But he would have prefered the development in terms of delay integral quadratic contsraints \cite{Yakubovich:02,Altshuller:04,Altshuller:11,Altshuller:13}; although such development is straightforward, we have not resolved some minor technical details, and prefer to retain the classical analysis with which we are more comfortable. In addition, Dmitry proposed the development in the more elegant framework of Fourier analysis on locally Abelian compact groups \cite{Rudin:62,Freedman:69}; for the time-being this will have to remain as an exercise for the reader. We miss working with Dmitry.
\begin{center}
WPH and JC.
\end{center}

\bibliographystyle{IEEEtran}

\bibliography{asym_bib_trans}

\begin{thebibliography}{10}
\providecommand{\url}[1]{#1}
\csname url@samestyle\endcsname
\providecommand{\newblock}{\relax}
\providecommand{\bibinfo}[2]{#2}
\providecommand{\BIBentrySTDinterwordspacing}{\spaceskip=0pt\relax}
\providecommand{\BIBentryALTinterwordstretchfactor}{4}
\providecommand{\BIBentryALTinterwordspacing}{\spaceskip=\fontdimen2\font plus
\BIBentryALTinterwordstretchfactor\fontdimen3\font minus
  \fontdimen4\font\relax}
\providecommand{\BIBforeignlanguage}[2]{{%
\expandafter\ifx\csname l@#1\endcsname\relax
\typeout{** WARNING: IEEEtran.bst: No hyphenation pattern has been}%
\typeout{** loaded for the language `#1'. Using the pattern for}%
\typeout{** the default language instead.}%
\else
\language=\csname l@#1\endcsname
\fi
#2}}
\providecommand{\BIBdecl}{\relax}
\BIBdecl

\bibitem{desoer75}
C.~A. Desoer and M.~Vidyasagar, \emph{Feedback Systems: Input-Output
  Properties}.\hskip 1em plus 0.5em minus 0.4em\relax Academic Press, 1975.

\bibitem{Rudin:53}
W.~Rudin, \emph{Principles of Mathematical Analysis}.\hskip 1em plus 0.5em
  minus 0.4em\relax McGraw-Hill, 1962.

\bibitem{Kolmogorov:57}
A.~N. Kolmogorov and S.~V. Fomin, \emph{Elements of the Theory of Functions and
  Functional Analysis}.\hskip 1em plus 0.5em minus 0.4em\relax Dover, 1957.

\bibitem{OShea67}
R.~O'Shea, ``An improved frequency time domain stability criterion for
  autonomous continuous systems,'' \emph{IEEE Transactions on Automatic
  Control}, vol.~12, no.~6, pp. 725 -- 731, 1967.

\bibitem{Zames68}
G.~Zames and P.~L. Falb, ``Stability conditions for systems with monotone and
  slope-restricted nonlinearities,'' \emph{SIAM Journal on Control}, vol.~6,
  no.~1, pp. 89--108, 1968.

\bibitem{Carrasco:EJC}
J.~Carrasco, M.~C. Turner, and W.~P. Heath, ``Zames-{F}alb multipliers for
  absolute stability: from {O}'{S}hea's contribution to convex searches,''
  \emph{European Journal of Control}, vol.~28, pp. 1 -- 19, 2016.

\bibitem{Vidyasagar78}
M.~Vidyasagar, \emph{Nonlinear Systems Analysis}.\hskip 1em plus 0.5em minus
  0.4em\relax Englewood Cliffs, NJ, USA: 720 Prentice-Hall, 1978.

\bibitem{Zames66}
G.~Zames, ``On the inout-output stability of time-varying nonlinear feedback
  systems. {P}art {II}: conditions involving circles in the frequency plane and
  sector nonlinearities,'' \emph{IEEE Transactions on Automatic Control},
  vol.~11, no.~3, pp. 465--476, 1966.

\bibitem{Altshuller:13}
D.~Altshuller, \emph{Frequency Domain Criteria for Absolute Stability: A
  Delay-integral-quadratic Constraints Approach}.\hskip 1em plus 0.5em minus
  0.4em\relax Springer, 2013.

\bibitem{Barabanov:00}
N.~E. Barabanov, ``The state space extension method in the theory of absolute
  stability,'' \emph{IEEE Transactions on Automatic Control}, vol.~45, no.~12,
  pp. 2335--2339, 2000.

\bibitem{Materassi11}
D.~Materassi and M.~Salapaka, ``A generalized {Z}ames-{F}alb multiplier,''
  \emph{IEEE Transactions on Automatic Control}, vol.~56, no.~6, pp.
  1432--1436, 2011.

\bibitem{HeathCDC15}
W.~P. {Heath}, J.~{Carrasco}, and D.~A. {Altshuller}, ``Stability analysis of
  asymmetric saturation via generalised {Z}ames-{F}alb multipliers,'' in
  \emph{2015 54th IEEE Conference on Decision and Control (CDC)}, 2015, pp.
  3748--3753.

\bibitem{Rantzer2001}
A.~Rantzer, ``Friction analysis based on integral quadratic constraints,''
  \emph{Int. J. Robust Nonlinear Control}, vol.~11, no.~7, pp. 645--652, 2001.

\bibitem{Chen:95}
X.~Chen and J.~T. Wen, ``Robustness analysis of {LTI} systems with structured
  incrementally sector bounded nonlinearities,'' in \emph{American Control
  Conference}, 1995.

\bibitem{Wang:TAC}
J.~Carrasco, W.~P. Heath, J.~Zhang, N.~S. Ahmad, and S.~Wang, ``Convex searches
  for discrete-time {Z}ames-{F}alb multipliers,'' \emph{IEEE Transactions on
  Automatic Control}, 2019, in press, doi:110.1109/TAC.2019.2958848.

\bibitem{Billingsley95}
P.~Billingsley, \emph{Probablity and measure, 3rd edition}.\hskip 1em plus
  0.5em minus 0.4em\relax Wiley, 1995.

\bibitem{Carrasco:12}
J.~Carrasco, W.~P. Heath, and A.~Lanzon, ``Factorization of multipliers in
  passivity and {IQC} analysis,'' \emph{Automatica}, vol.~48, no.~5, pp.
  909--916, 2012.

\bibitem{megretski97}
A.~{Megretski} and A.~{Rantzer}, ``System analysis via integral quadratic
  constraints,'' \emph{IEEE Transactions on Automatic Control}, vol.~42, no.~6,
  pp. 819--830, 1997.

\bibitem{Yakubovich:02}
V.~A. Yakubovich, ``Popov's method and its subsequent development,''
  \emph{European Journal of Control}, vol.~8, no.~3, pp. 200--208, 2002.

\bibitem{Altshuller:04}
D.~A. Altshuller, A.~V. Proskurnikov, and V.~A. Yakubovich, ``Frequency-domain
  criteria for dichotomy and absolute stability for integral equations with
  quadratic constraints involving delays,'' \emph{Doklady Mathematics},
  vol.~70, no.~3, pp. 998--1002, 2004.

\bibitem{Altshuller:11}
D.~A. Altshuller, ``Delay-integral-quadratic constraints and stability
  multipliers for systems with {MIMO} nonlinearities,'' \emph{IEEE Transactions
  on Automatic Control}, vol.~56, no.~4, pp. 738--747, 2011.

\bibitem{Carrasco15}
J.~{Carrasco} and P.~{Seiler}, ``Integral quadratic constraint theorem: A
  topological separation approach,'' in \emph{2015 54th IEEE Conference on
  Decision and Control (CDC)}, 2015, pp. 5701--5706.

\bibitem{HeathECC15}
W.~P. Heath and J.~Carrasco, ``Global asymptotic stability for a class of
  discrete-time systems,'' in \emph{Proceedings of the European Control
  Conference}, 2015.

\bibitem{Heath14}
W.~P. Heath, J.~Carrasco, and M.~de~la Sen, ``Second-order counterexamples to
  the discrete-time {K}alman conjecture,'' \emph{Automatica}, vol.~60, pp.
  140--144, 2015.

\bibitem{Groff:19}
L.~B. {Groff}, J.~M. {Gomes da Silva}, and G.~{Valmorbida}, ``Regional
  stability of discrete-time linear systems subject to asymmetric input
  saturation,'' in \emph{2019 58th IEEE Conference on Decision and Control
  (CDC)}, 2019.

\bibitem{Carrasco:13}
J.~Carrasco, W.~P. Heath, and A.~Lanzon, ``Equivalence between classes of
  multipliers for slope-restricted nonlinearities,'' \emph{Automatica},
  vol.~49, no.~6, pp. 1732--1740, 2013.

\bibitem{Willems:68}
J.~{Willems} and R.~{Brockett}, ``Some new rearrangement inequalities having
  application in stability analysis,'' \emph{IEEE Transactions on Automatic
  Control}, vol.~13, no.~5, pp. 539--549, October 1968.

\bibitem{Willems:71}
J.~C. Willems, \emph{The Analysis of Feedback Systems}.\hskip 1em plus 0.5em
  minus 0.4em\relax The MIT Press, 1971.

\bibitem{Heath:07}
W.~P. Heath and A.~G. Wills, ``Zames–{F}alb multipliers for quadratic
  programming,'' \emph{IEEE Transactions on Automatic Control}, vol.~52,
  no.~10, pp. 1948--1951, 2007.

\bibitem{Petsagkourakis:20}
P.~Petsagkourakis, W.~P. Heath, J.~Carrasco, and C.~Theodoropoulos, ``Robust
  stability of barrier-based model predictive control,'' \emph{IEEE
  Transactions on Automatic Control}, 2020, in press,
  doi:10.1109/TAC.2020.3010770.

\bibitem{Lessard:16}
L.~Lessard, B.~Recht, and A.~Packard, ``Analysis and design of optimization
  algorithms via integral quadratic constraints,'' \emph{SIAM Journal on
  Optimization}, vol.~26, no.~1, pp. 57--95, 2016.

\bibitem{Michalowsky:20}
S.~Michalowsky, C.~Scherer, and C.~Ebenbauer, ``Robust and structure exploiting
  optimisation algorithms: an integral quadratic constraint approach,''
  \emph{International Journal of Control}, 2020, in press,
  doi:10.1080/00207179.2020.1745286.

\bibitem{Wang:18}
S.~Wang, J.~Carrasco, and W.~P. Heath, ``Phase limitations of {Z}ames-{F}alb
  multipliers,'' \emph{IEEE Transactions on Automatic Control}, vol.~63, no.~4,
  pp. 947--959, 2018.

\bibitem{Zhang:20}
J.~Zhang, J.~Carrasco, and W.~P. Heath, ``Duality bounds for discrete-time
  {Z}ames-{F}alb multipliers,'' \emph{arXiv}, 2020, arXiv:2008.11975.

\bibitem{Rudin:62}
W.~Rudin, \emph{Fourier Analysis on Groups}.\hskip 1em plus 0.5em minus
  0.4em\relax Interscience Publishers, 1962.

\bibitem{Freedman:69}
M.~I. Freedman, P.~L. Falb, and G.~Zames, ``A {H}ilbert space stability theory
  over locally compact {A}belian groups,'' \emph{SIAM Journal of Control},
  vol.~7, pp. 479--495, 1968.

\end{thebibliography}


%




\ifCLASSOPTIONcaptionsoff
  \newpage
\fi

\begin{IEEEbiography}[{\includegraphics[width=1in,height=1.25in,clip,keepaspectratio]{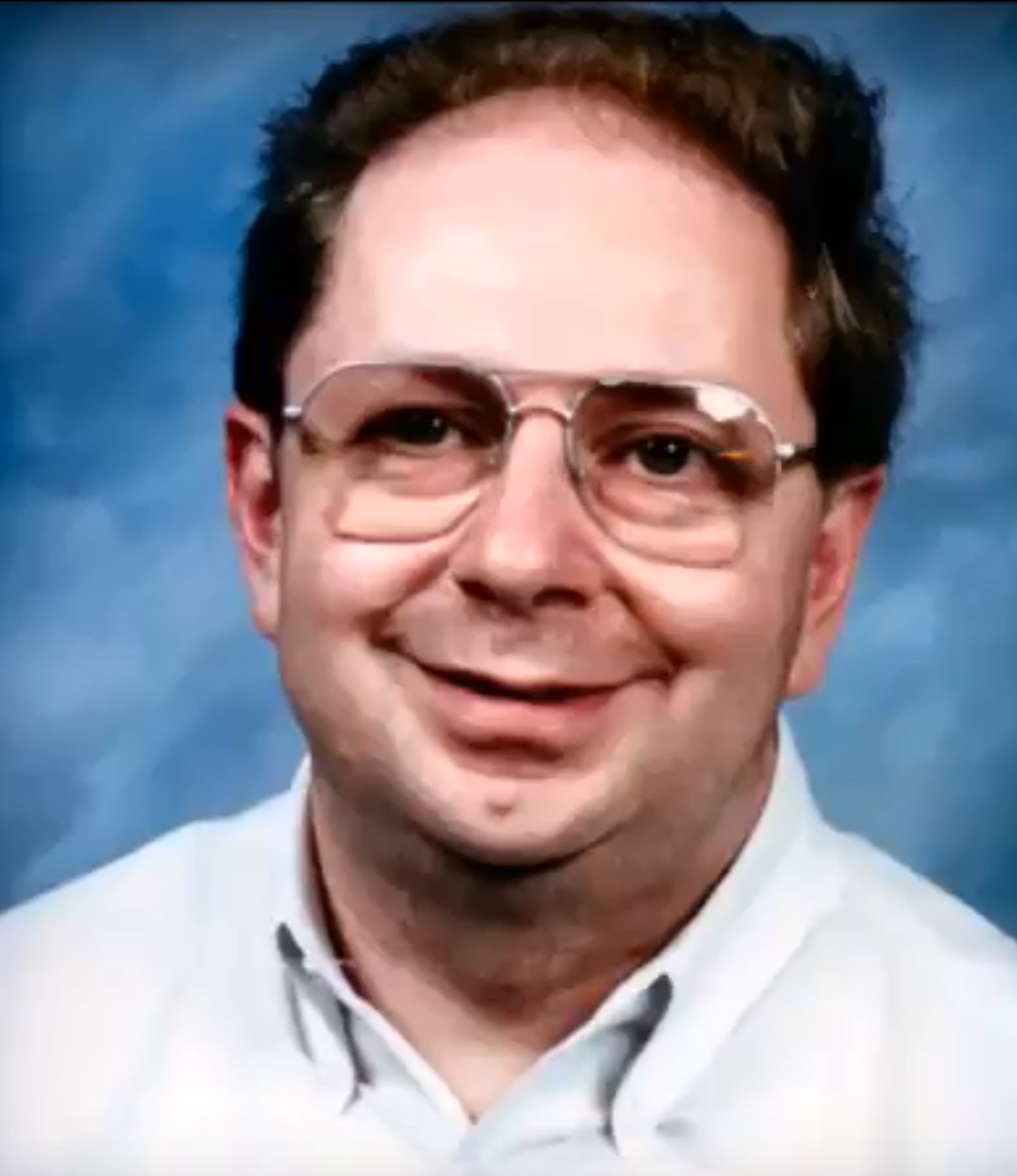}}]{Dmitry Alexander Altshuller} was born on January 16, 1961. He graduated from high school in Leningrad (now St Petersburg) with Spanish as a foreign language. He emigrated with his family to the USA in 1979.

He received the B.Math. degree from the University of Minnesota, Minneapolis, in 1982, the M.S. degree in Control Systems Engineering from Washington University, St. Louis, MO, in 1995, and the Ph.D. (Kandidat Nauk) degree in Physics and Mathematics (with specialty in Theoretical Cybernetics) from St. Petersburg State University, St. Petersburg, Russia, in 2004.  His supervisor was Vladimir Yakubovich who had been a childhood friend of his father. His thesis was entitled  “Absolute Stability of Control Systems with Nonstationary Nonlinearities”. He published an extended version under the title “Frequency Domain Criteria for Absolute Stability” with Springer in 2013.

He was a Member of Technical Staff at Lucent Technologies, worked as a Control Systems Engineer for MH Systems, Inc and as a Staff Mathematician for Scientific Applications and Research Associates, Inc. Latterly he worked with Crane Aerospace, Parker Aerospace  and Dassault Systems.
 
He was an author or  co-author of over 30 published and/or presented papers, including a partial solution of one the problems described in the renowned book Unsolved Problems in Mathematical Systems and Control Theory (Blondel and Megretski). His research (done mostly on his own time) involved various aspects of nonlinear control systems, including stability and optimal control. He was a Senior Member of the IEEE and served as Chair of the Orange County Aerospace and Electronic Systems Society Chapter 2012-2013. He was a member of the Society for Industrial and Applied Mathematics and of the International Physics and Control Society (IPACS). He served on program committees for several conferences.

He suffered a heart attack while on vacation  in Key Largo, FL and passed away May 26, 2017, one day after celebrating his 21st wedding anniversary to Mary Altshuller. He loved scuba diving and played some chess. He also earned a sport pilot's license and was planning on becoming instrument-rated at some point.

He is survived by his wife Mary, his parents Drs. Mark and Elena Altshuller and a daughter and stepson.
\end{IEEEbiography}

\end{document}